\documentclass[12pt, reqno]{amsart}
\usepackage{amsmath, amsthm, amscd, amsfonts, amssymb, graphicx, color, enumerate, xcolor,  mathrsfs, latexsym}
\usepackage[bookmarksnumbered, colorlinks, plainpages]{hyperref}
\hypersetup{colorlinks=true,linkcolor=red, anchorcolor=green, citecolor=cyan, urlcolor=red, filecolor=magenta, pdftoolbar=true}

\definecolor{TITLE}{rgb}{0.0,0.0,1.0}
\definecolor{AUTHOR1}{rgb}{0.00,0.59,0.00}
\definecolor{AUTHOR2}{rgb}{0.50,0.00,1.00}
\definecolor{SECTION}{rgb}{0.50,0.00,1.00}
\definecolor{FOOTTITLE}{rgb}{0.00,0.50,0.75}
\definecolor{THM}{rgb}{0.7,0.3,0.3}
\definecolor{SEC}{rgb}{0.6,0.1,.5}

\makeatletter
\@namedef{subjclassname@2010}{%
  \textup{2010} Mathematics Subject Classification}
\makeatother
\newtheorem{theorem}{{\color{THM} Theorem}}[section]
\newtheorem{lemma}[theorem]{{\color{THM}Lemma}}
\newtheorem{proposition}[theorem]{{\color{THM}Proposition}}
\newtheorem{corollary}[theorem]{{\color{THM}Corollary}}
\theoremstyle{definition}
\newtheorem{definition}[theorem]{{\color{THM}Definition\ }}

\numberwithin{equation}{section}

\numberwithin{equation}{section}

\label{sec.sub.gdn}

\newcommand{\N}{\mathcal N}
\newcommand{\V}{\mathcal V}

\begin{document}
\title[More on  the Arens regularity of $B(X)$]{More on  the Arens regularity of  $B(X)$ }
\author[R. Faal]{R. Faal}
\author[H.R. Ebrahimi Vishki]{H.R. Ebrahimi Vishki}
\address{ Department of Pure Mathematics, Ferdowsi University of Mashhad, P.O. Box 1159, Mashhad 91775, IRAN.}
\email{faal.ramin@yahoo.com}
\address{ Department of Pure Mathematics and  Center of Excellence in Analysis on Algebraic Structures (CEAAS), Ferdowsi University of Mashhad, P.O. Box 1159, Mashhad 91775, IRAN.}
\email{vishki@um.ac.ir}
\begin{abstract}  We focus on  a question raised by M. Daws [Bull. London Math. Soc.  {36}  (2004), 493-503] concerning the Arens regularity of $B(X)$, the algebra of operators on a Banach space $X.$ In this respect, among other things, we   show that $B(X)$ is Arens regular if and only if $X$ is ultra-reflexive.
 \end{abstract}

\maketitle

\section{Introduction}\label{Intro}
The second dual $A^{**}$ of a Banach algebra $A$ can be made into a Banach algebra with two, in general different, (Arens) products, each extending the original product of $A$ \cite{A}.    A Banach algebra $A$ is said to be  Arens regular when the Arens products coincide. For example, every $C^*-$algebra is Arens regular \cite{CY}. For an explicit description of the properties of these products and the notion of Arens regularity one may consult with \cite{Da}.

For the Banach algebra $B(X)$, bounded operators on a Banach space $X$,  Daws  showed that, if $X$ is super-reflexive then  $B(X)$ is Arens regular; \cite[Theorem 1]{D}. He also conjectured  the validity of the converse. To the best of our knowledge, it seems that this has not be solved yet. It has been, however, known that the Arens regularity of $B(X)$ necessities the reflexivity of $X,$ (for a proof see \cite[Theorems 2, 3]{Y} or \cite[Theorem 2.6.23]{Da}).

 In Section 2 we provide some preliminaries related to ultrapowers and super-reflexivity. In Section 3 we prove Theorem \ref{tau} from which we derive that   the reflexivity of $X$ is equivalent to the $wo-$compactness of {\rm Ball}$(B(X)).$ This motivates to  introduce the notion of  ultra-reflexive space and compare it with the super-reflexivity. Section 4 is devoted to the main result of the paper (Theorem \ref{main}) stating that: $B(X)$ is Arens regular if and only if $X$ is ultra-reflexive.
\section{Preliminaries}\label{2}
Let $X$ be a Banach space, $I$ be an indexing set and let $\mathcal{U}$ be an ultrafilter on $I.$ We define the  ultrapower $X_\mathcal{U}$ of $X$ with respect to $\mathcal{U},$ by the quotient space
\[X_\mathcal{U}=\ell^\infty(X,I)/\N_\mathcal{U}, \]
where $\ell^\infty(X,I)$ is the Banach space
\[\ell^\infty(X,I)=\{(x_{\alpha})_{\alpha\in I}\subseteq X: \|(x_{\alpha})\|=\sup_{\alpha\in I}\|x_{\alpha}\|<\infty\},\]
and $\N_\mathcal{U}$ is the closed subspace \[\N_\mathcal{U}=\{(x_{\alpha})_{\alpha\in I}\in\ell^\infty(X,I):  \lim_{\mathcal{U}}\|x_{\alpha}\|=0\}.\]
Then the norm $\|(x_{\alpha})\|_\mathcal{U}:=\lim_\mathcal{U}\|x_{\alpha}\|$ coincides with the quotient norm.
We can identify $X$ with a closed subspace of $X_\mathcal{U}$ via the canonical isometric embedding $X\hookrightarrow X_\mathcal{U},$ sending $x\in X$ to the constant family $(x).$  Ample information  about  ultrapowers can be found in \cite{H}.

A Banach space $X$ is called super-reflexive if every  finitely presentable Banach space  in $X$ is reflexive. We recall that a Banach space $Y$ is said to be finitely representable in $X$ if each finite dimensional subspace  of $Y$ is $(1+\epsilon)-$isomorphic to some subspace of $X$, for each $\epsilon>0.$ For example, every Banach space is finitely representable in $c_0,$ and every finitely representable Banach space in $\ell^2$ is a Hilbert space. In the language of ultrapowers, it has been shown  that $Y$ is finitely representable in $X$ if and only if $Y$ is isometrically isomorphic to a subspace of $X_\mathcal{U}$ for some ultrafilter $\mathcal{U}$ on $X;$ \cite[Theorem 6.3]{H}. It follows that a Banach space  is super-reflexive if and only if all of its ultrapowers are reflexive. It has also proved that $X$ is super-reflexive if and only if $X^*$ is super-reflexive \cite{H}.

As it has been shown in \cite[Section 7]{H}, there is a canonical isometry  $J :(X^*)_\mathcal{U}\rightarrow (X_\mathcal{U})^*$ defined
by the rule \[\langle J((f_\alpha)_\mathcal{U}) , (x_\alpha )_\mathcal{U} \rangle =\lim_\mathcal{U} \langle f_\alpha , x_\alpha \rangle\qquad  ((f_\alpha )_\mathcal{U}\in (X^*)_\mathcal{U},  (x_\alpha )_\mathcal{U}\in X_\mathcal{U}),\] which  is a surjection if and only if $X_\mathcal{U}$ is reflexive (where $\mathcal U$ is countably incomplete). In particular,  when $X$ is super-reflexive then $J$ is an isometric isomorphism.

As the {\rm Ball}$(X^{**})$  is $w^*-$compact, one  can define a norm-decreasing map $\sigma : X_\mathcal{U}\rightarrow X^{**}$ by
\[ \sigma ( (x_\alpha)_\mathcal{U}) = w^*-\lim_\mathcal{U} \kappa_X (x_\alpha), \qquad ((x_\alpha)_\mathcal{U}\in X_\mathcal{U}), \]
where $\kappa_X $ is the canonical embedding of $X$ into $X^{**}$. We quote the next result from \cite{H} which will be needed in the subsequent sections.
\begin{proposition} [{\cite[Proposition 6.7]{H}}]\label{add}
Let $X$ be a Banach space. Then there exist an ultrafilter $\mathcal{U}$ and a linear isometric embedding $K : X^{**}\rightarrow X_\mathcal{U}$ such that $\sigma \circ K$ is the identity on $X^{**}$ and $K\circ\kappa_X$ is the canonical embedding of $X$ into $X_\mathcal{U}.$ Thus $K\circ\sigma$ is a norm-1 projection of $X_\mathcal{U}$ onto $K(X^{**})$.
\end{proposition}
It is worthwhile mentioning that the  ultrafilter $\mathcal{U}$ used in the above proposition is countably incomplete. Indeed, $\mathcal{U}$ is the ultrafilter inducing by refining the order filter on the set \[I=\lbrace (M,N,\varepsilon) : M\subseteq X^{**} \ {\rm is\ finite},\ N\subseteq X^* \ {\rm is\ finite},\  \varepsilon>0\rbrace.\]
   Set $I_n= I_{(M, N, \frac{1}{n})}=\lbrace (M_0, N_0, \varepsilon)\in I : \  M\subseteq M_0,  N\subseteq N_0\ {\rm and\ } \varepsilon\leq \frac{1}{n}\rbrace.$  Then $I_{n+1}\subseteq I_n$ and $\cap_{n=1}^\infty I_n=\emptyset,$  so  $\mathcal{U}$ is countably incomplete.\\

There are several criterions for the Arens regularity of a Banach algebra, among which, we quote the following that will  be frequently used  in the sequel, (for a proof see \cite {Da, D}).
\begin{proposition}\label{arens}
For every Banach algebra $A$ the following assertions are equivalent.
\begin{enumerate}[\hspace{1em}\rm (1)]
\item $A$ is Arens regular.
\item For each $\lambda\in A^*$ the operator $a\mapsto\lambda\cdot a: A\longrightarrow A^*$ is weakly compact.
\item For each $\lambda\in A^*$  there exist a reflexive  space $Z$ and bounded linear maps $\phi:A\rightarrow Z$ and $\psi:A\rightarrow Z^*$ such that $\langle\lambda, ab\rangle=\langle\psi(a),\phi(b)\rangle$ for all $a, b\in A.$
\end{enumerate}
\end{proposition}
We remark  that, in Proposition \ref{arens}, one can choose $Z$ and $\phi$ so that $\|\lambda\|\leq\|\phi\|.$\\
\section{Weak operator compactness and reflexivity}\label{3}
Let $X$ and $Y$ be two Banach spaces and let $\tau$ be a locally convex topology on $Y$ induced by a separating family $\{p_\gamma\}_{\gamma\in\Gamma}$ of semi-norms. Then $\tau$ induces a $\tau o-$topology on $B(X,Y)$ which is induced by the family  $\{x\otimes p_\gamma\}_{x\in X, \gamma\in\Gamma}$ of semi-norms, where  $(x\otimes p_\gamma)(T)= p_\gamma(T(x)),$  $(x\in X, T\in B(X,Y)).$ We thus have  $T_\alpha\xrightarrow{\tau o} T$ if and only if $T_\alpha (x)\xrightarrow{\tau }T(x),$ for each $x\in X$. For example, in the case when $\tau=w$ is the weak topology on $Y$ then the $\tau o-$topology on $B(X,Y)$ is nothing but the weak operator topology ($wo-$topology) on $B(X,Y).$
The next result establishes a close relation between  {\rm Ball}$(Y)$ and {\rm Ball}$(B(X,Y))$.
\begin{theorem}\label{tau}
Let $X$ and $Y$ be two Banach spaces  and let $\tau$ be a locally convex topology  on $Y$  which is weaker than the norm topology and induced by  the  semi-norms $\lbrace p_\gamma \rbrace$ such that   $\| y\|\leq \sup_{\| p_\gamma \|\leq 1} |p_\gamma (y)|,$  for all $y\in Y.$ Then {\rm Ball}$(Y)$ is $\tau-$compact if and only if {\rm Ball}$(B(X,Y))$  is  $\tau o-$compact.
\end{theorem}
\begin{proof}
Let {\rm Ball}$(B(X,Y))$ is $ \tau o-$compact and fix a $f\in X^*$ with $\|f\|=1$  and $f(x_0)=1$, for some $x_0\in X$. Then the  operator $\Psi:Y\rightarrow B(X,Y)$ where $\Psi(y)(x)=f(x)y$ is an isometry. Moreover $y_\alpha\xrightarrow{\tau} y$ if and only if $\Psi(y_\alpha)\xrightarrow{\tau o}\Psi(y)$. If $\{y_\alpha\}$ is  a net in {\rm Ball}$(Y)$ then $\{\Psi(y_\alpha)\}$ is a net in {\rm Ball}$(B (X,Y)),$  so it has a $\tau o-$convergent subnet say $\{\Psi(y_{\alpha_\beta})\}$. Therefore  $\{y_{\alpha_\beta}\}$ is $\tau-$convergent in $Y$, that is,  {\rm Ball}$(Y)$ is $\tau-$compact.

For the converse we define the operator $\Phi:B(X,Y)\rightarrow \Pi_{x\in S_X} Y$ with the rule $\Phi(T)=(Tx)_{x\in S_X}$. Obviously $\Phi$ is one to one, and  $T_\alpha\xrightarrow{\tau o}T$ if and only if $\Phi(T_\alpha)\xrightarrow{\Pi \tau}\Phi(T)$.
Let $\{T_\alpha\}$ be a net in {\rm Ball}$(B(X,Y))$ then  $\{\Phi(T_\alpha)\}$ is a net  in $\Pi_{x\in S_X} \texttt{\rm Ball}(Y)$. By Tychonoff  theorem $\Pi_{x\in S_X} \texttt{\rm Ball}(Y)$ is compact in the  product $\tau-$topology, so $\{\Phi(T_\alpha)\}$ enjoys a subnet $\{\Phi(T_{\alpha_\beta})\}$ which is  convergent in the product $\tau-$topology. We then can  define  an operator $T:X\to X$ by $T(x)=\tau-\lim_\beta T_{\alpha_\beta}(x).$  For each  $x\in S_X,$ since
$|p_\gamma (T_{\alpha_\beta}(x))|\leq \| p_\gamma\| \ \|T_{\alpha_\beta}(x)\|,$
so $
|p_\gamma (T(x))|\leq \| p_\gamma\| \liminf \|T_{\alpha_\beta}(x)\|.$ It follows that $
\|T(x) \|\leq \sup_{\| p_\gamma \|\leq 1} |p_\gamma (T(x))|\leq \liminf \|T_{\alpha_\beta}(x)\|\leq 1$. We thus have $\|T\|\leq 1$ and $T_{\alpha_\beta}\xrightarrow{\tau o}T,$ as claimed.
\end{proof}
As an immediate consequence  we present the next result, part (1) of  which will be frequently used in the sequel.
\begin{proposition}\label{wo}
Let $X$ be a Banach space. Then
\begin{enumerate}[\hspace{1em}\rm (1)]
\item {\rm Ball}$(B(X))$ is $wo-$compact if and only if  $X$ is reflexive.
\item {\rm Ball}$(B(X))$ is $so-$compact if and only if  $X$ is finite dimensional.
\item  {\rm Ball}$(X^*)$ is  $w^*-$compact (Banach-Alaoglu).
\end{enumerate}
\end{proposition}
\begin{proof}
For (1) (resp. (2)) we use Theorem \ref{tau} for $Y=X$ with $\tau$ as the weak (resp. norm) topology.  For (3) we use Theorem \ref{tau} for $Y=\mathbb C$ with $\tau$ as the usual topology.
\end{proof}
\section{Arens regularity of $B(X)$ and ultra-reflexivity}\label{4}
We commence with the   next key lemma that will be frequently used in the sequel.
\begin{lemma}\label{1}
If $X$ is reflexive then there exist an (countably incomplete) ultrafilter $\mathcal{U}$ such that every $\lambda\in B(X)^*$ can be identified with $x_\mathcal{U}{\otimes}f_\mathcal{U}$ for some $x_\mathcal{U}\in \ell^2(X)_\mathcal{U}, f_\mathcal{U}\in\ell^2(X^*)_\mathcal{U}.$
\end{lemma}
\begin{proof}
The reflexivity of $X$ implies that $B(X)^*\cong(X\widehat{\otimes}X^*)^{**}. $ By Proposition \ref{add} there exist an ultrafilter $\mathcal{U}$ and a linear isometric embedding  \[(X\widehat{\otimes}X^*)^{**}\hookrightarrow (X\widehat{\otimes}X^*)_\mathcal{U},\]
such that the composition  $(X\widehat{\otimes}X^*)\hookrightarrow(X\widehat{\otimes}X^*)^{**}\hookrightarrow (X\widehat{\otimes}X^*)_\mathcal{U}$ coincides with the canonical embedding   $(X\widehat{\otimes}X^*)\hookrightarrow (X\widehat{\otimes}X^*)_\mathcal{U}.$
We consider arbitrary $\lambda\in B(X)^*\hookrightarrow (X\widehat{\otimes}X^*)_\mathcal{U}$ then $\lambda =(\lambda_\alpha )_\mathcal{U}$ which $\lambda_\alpha =\sum_{n=1}^\infty x_n^{\alpha} \otimes f_n^{\alpha}$, $\| \lambda_\alpha \| \leq \sum_{n=1}^\infty \| x_n^{\alpha} \|  \ \| f_n^{\alpha} \|\leq \| \lambda \| +1$ and $\|x_n^{\alpha} \| = \|f_n^{\alpha} \|$ for each $(\alpha,n)\in I\times \mathbb{N}$.

For each $\alpha,$ put $x^{\alpha} =(x^{\alpha}_n)_{n\in\mathbb N}$ and $f^{\alpha} =(f^{\alpha}_n)_{n\in \mathbb N},$ then $x^{\alpha}\in\ell^2(X)$ and $f^{\alpha}\in\ell^2(X^*),$ indeed
\[\| x^{\alpha} \|_2^2=\sum_{n=1}^\infty \| x^{\alpha}_n \|^2 = \sum_{n=1}^\infty \| x^{\alpha}_n \| \ \|f^{\alpha}_n \|\leq \| \lambda \|+1.\]
 Now set $x_\mathcal{U}=(x^{\alpha})_\mathcal{U}$ and $f_\mathcal{U}= (f^{\alpha})_\mathcal{U}$ then  clearly $x_\mathcal{U}\in \ell^2(X)_\mathcal{U}$ and $f_\mathcal{U}\in \ell^2(X^*)_\mathcal{U}$. Then  $\lambda =x_\mathcal{U}\otimes f_\mathcal{U}$, indeed
\[\lambda (T) =(\lambda_\alpha)_\mathcal{U}(T)=\lim_\mathcal{U}\lambda_\alpha(T) =\lim_\mathcal{U} \sum_{n=1}^\infty f_n^{\alpha}(Tx_n^{\alpha})=( (x^{\alpha})_\mathcal{U}\otimes (f^{\alpha})_\mathcal{U})(T)=(x_\mathcal{U}\otimes f_\mathcal{U})(T).\]
\end{proof}
We recall that the super-reflexivity of $X$ is equivalent to that  of  $\ell^2(X)$, the Banach space of all $2-$summable sequences in $X$, (see \cite[Proposition 4]{D}).   So $X$ is super-reflexive if and only if ${\rm Ball}(\ell^2(X)_\mathcal{U})$ is weakly compact, or equivalently, by  Proposition \ref{wo},   {\rm Ball}$(B(\ell^2(X)_\mathcal{U}))$ is  $wo-$compact for every ultrafilter $\mathcal{U}.$ This motivates to introduce  the notion of ultra-reflexivity in the next definition.
\begin{definition} A Banach space $X$ is called ultra-reflexive if ${\rm Ball}(B(X))(x_\mathcal{U})$ is weakly compact for every ultra-filter $\mathcal{U}$ and each $x_\mathcal{U}\in\ell^2(X)_\mathcal{U}.$
\end{definition}
It is obvious that every ultra-reflexive space $X$ is reflexive. Indeed, for each non-zero $x\in X$, $B(X)(x)=X.$ It is also worth to note that, if $X$ is super-reflexive then  $X$ is ultra-reflexive; see Corollary \ref{su}. Therefore ultra-reflexivity lies between reflexivity and super-reflexivity.

We are now ready to prove our main result characterizing the Arens regularity of $B(X)$ in terms of the ultra-reflexivity of $X.$ Before proceedin, we quote the next technical lemma from \cite{DFJP} which will be used in the proof of Theorem \ref{main}.

\begin{lemma} [{\cite[Lemma 1]{DFJP}}]\label{reflex}
Let $X$ be a Banach space and $W\subseteq X$ be a bounded, symmetric and convex subset. For each $n\in\mathbb N$ let  the norm $\| \cdot\|_n$ denote the gauge of $U_n=2^n W+2^{-n}{\rm Ball}(X).$  Set $Y=\lbrace x\in X : |||x|||=(\sum_{n=1}^\infty \|x\|_n^2)^{\frac{1}{2}}<\infty\rbrace,$ then
\begin{enumerate}[\hspace{1em}\rm (i)]
\item
$W\subseteq {\rm Ball}(Y)$.
\item
$(Y, |||\cdot|||)$ is a Banach space and the identity embedding $j:Y\to X$ is bounded.
\item
$j^{**}:Y^{**}\to X^{**}$ is one to one and $(j^{**})^{-1}(X)=Y$.
\item
$Y$ is reflexive if and only if $W$ is weakly relatively compact.
\end{enumerate}
\end{lemma}

\begin{theorem}\label{main} For a Banach space $X$ the following assertions are equivalent.
\begin{enumerate}[\hspace{1em}\rm (a)]
\item
$B(X)$ is Arens regular.
\item
$f_\mathcal{U}\circ {\rm Ball}(B(X))$ is $w^*-$compact  for  every ultrafilter $\mathcal{U}$ and each $f_\mathcal{U}\in{\ell^2(X)_\mathcal{U}}^*.$
\item
$f_\mathcal{U}\circ {\rm Ball}(B(X))$ is $w-$compact  for  every ultrafilter $\mathcal{U}$ and each $f_\mathcal{U}\in{\ell^2(X)_\mathcal{U}}^*.$
\item
$X$ is ultra-reflexive.
\end{enumerate}
\end{theorem}
\begin{proof}
(a)$\Rightarrow$ (b):
Let  $B(X)$ be  Arens regular,   $\mathcal U$ be an ultrafilter, $f_\mathcal{U}\in {\ell^2(X)_\mathcal{U}}^*$ and $x\in X.$  They  induce the functional $x\otimes f_\mathcal{U}\in B(X)^*. $

Suppose that $\{T_\alpha\}$ is  a net in {\rm Ball}$(B(X)).$ As $X$ is reflexive \cite[Theorem 2.6.23]{Da}, by Proposition \ref{wo}  {\rm Ball}$(B(X))$ is $wo-$compact, so $\{T_\alpha\}$ enjoys a subnet  $\{T_{\alpha_\beta}\}$ such that $T_{\alpha_\beta}\xrightarrow{wo} T_0,$ for some $T_0\in{\rm Ball}(B(X)).$ This implies that
 \[(x\otimes f_\mathcal{U})\cdot T_{\alpha_\beta}\xrightarrow{w^*}(x\otimes f_\mathcal{U})\cdot T_0.\]
By Proposition \ref{arens}, the Arens regularity of $B(X)$ also implies the weak compactness of the operator \[T\mapsto (x\otimes f_\mathcal{U})\cdot T:B(X)\rightarrow B(X)^*,\qquad (T\in B(X)),\]
from which we get
\[(x\otimes f_\mathcal{U})\cdot T_{\alpha_\beta}\xrightarrow{w}(x\otimes f_\mathcal{U})\cdot T_0,\]
for some new subnet $\{T_{\alpha_\beta}\}$ of itself.

Fix $x_\mathcal{U}\in\ell^2(X)_\mathcal{U}$  and let $S_\mathcal{U}\in B(\ell^2(X)_\mathcal{U})$ be so $S_\mathcal{U}(x)=x_\mathcal{U}.$

Define $F_{S_\mathcal{U}}:(x\otimes f_\mathcal{U})\cdot B(X)\to \mathbb{C}$ by \[F_{S_\mathcal{U}}\left((x\otimes f_\mathcal{U})\cdot T\right)=f_\mathcal{U}(T(S(x))), \qquad (T\in B(X)).\]
Then $F_{S_U}$ is linear and bounded.
 Indeed
\begin{eqnarray*}
\| F_{S_\mathcal{U}}\| &=&\sup_{\|(x\otimes f_\mathcal{U})\cdot T  \|\leq 1} | F_{S_\mathcal{U}}\left((x\otimes f_\mathcal{U})\cdot T \right) | \\
&=&\sup_{\| f_\mathcal{U}\circ T  \|\leq 1} | f_\mathcal{U}(T\circ S_\mathcal{U}(x))| \\
&\leq& \sup_{\| f_\mathcal{U}\circ T  \|\leq 1}  \ \| f_\mathcal{U}\circ T\| \ \| x\| \  \ \| S_\mathcal{U}\| \\
&\leq& \| S_\mathcal{U}\| .
\end{eqnarray*}
So we can extend $ F_{S_\mathcal{U}}$ to an element $ F_{S_U}\in B(X)^{**}$ with the same norm. We thus get
\begin{eqnarray*}
(f_\mathcal{U}\circ T_{\alpha_\beta})(x_\mathcal{U})&=&f_\mathcal{U}(T_{\alpha_\beta}(S_\mathcal{U}(x)))\\
&=&F_{S_\mathcal{U}}((x\otimes f_\mathcal{U})\cdot T_{\alpha_\beta})\\
&\longrightarrow& F_{S_\mathcal{U}}((x\otimes f_\mathcal{U})\cdot T_0)=(f_\mathcal{U}\circ T_0)(x_\mathcal{U}).
\end{eqnarray*}
This implies that $(f_\mathcal{U}\circ T_{\alpha_\beta})\xrightarrow{w^*} (f_\mathcal{U}\circ T_0)$, so $f_\mathcal{U}\circ {\rm Ball}(B(X))$  is $w^*-$compact in ${\ell^2(X)_\mathcal{U}}^*.$

(b)$\Rightarrow$ (c):
Fix an ultrafilter $\mathcal{U}$ and ${f_\mathcal{U}\in\ell^2(X)_\mathcal{U}}^*.$ By Proposition \ref{add}, there exists an embedding ${\ell^2(X)_\mathcal{U}}^{**}\hookrightarrow \ell^2(X)_{\mathcal{U}\times\mathcal{V}}$, for some ultrafilter $\V.$ We also consider the identification $f_\mathcal{U}\mapsto(f_\mathcal{U})_\mathcal{V}: {\ell^2(X)_\mathcal{U}}^*\hookrightarrow {\ell^2(X)_{\mathcal{U}\times\mathcal{V}}}^*.$ By (b) $(f_\mathcal{U})_\mathcal{V}\circ {\rm Ball}(B(X))$ is $w^*-$compact in ${\ell^2(X)_{\mathcal{U}\times\mathcal{V}}}^*$. Regarding the above identifications imply  that
$f_\mathcal{U}\circ {\rm Ball}(B(X))$ is $w-$compact in ${\ell^2(X)_\mathcal{U}}^*$.

(c)$\Rightarrow$ (d):
 First note that since $B(X)$ is isometrically $\ast-$anti-isomorphic to $B(X^*)$, the Arens regularity of $B(X)$ implies that of $B(X^*)$. Now using (c) for $x_\mathcal{U}\circ {\rm Ball}(B(X^*))$, the identification  ${\rm Ball}(B(X))(x_\mathcal{U})\cong x_\mathcal{U}\circ {\rm Ball}(B(X^*))$ implies that  ${\rm Ball}(B(X))(x_\mathcal{U})$ is weakly compact in $\ell^2(X)_\mathcal{U}$ for each $x_\mathcal{U}\in\ell^2(X)_\mathcal{U}$.

(d)$\Rightarrow$ (a): Suppose that $X$ is ultra-reflexive, then $X$ must be reflexive, so by Lemma \ref{1}, each element  $\lambda\in B(X)^*$ has the  tensorial form $\lambda =x_\mathcal{U}\otimes f_\mathcal{U}$ for some $x_\mathcal{U}\in\ell^2(X)_\mathcal{U}, {f_\mathcal{U}\in\ell^2(X)_\mathcal{U}}^*$. To prove $B(X)$ is Arens regular, we use Proposition \ref{arens}. For this we first apply Lemma \ref{reflex} for the subset $W={\rm Ball}(B(X))(x_\mathcal{U})$ of $\ell^2(X)_\mathcal{U}.$ It induces a reflexive subspace $Y_{x_\mathcal{U}}$ of $\ell^2(X)_\mathcal{U}$ such that $W\subseteq {\rm Ball}(Y_{x_\mathcal{U}})$ and that the identity embedding $j:Y_{x_\mathcal{U}}\longrightarrow \ell^2(X)_\mathcal{U}$ is bounded. Now we define  $\phi:B(X)\rightarrow Y_{x_\mathcal{U}} $ and $\psi: B(X)\rightarrow {Y_{x_\mathcal{U}}}^*$ with the rules $\phi(T)=T(x_\mathcal{U})$ and  $\psi(T)=f_\mathcal{U}\circ T\circ j,$ respectively. Then a  direct verification reveals that  $\phi, \psi$ are bounded linear  mappings satisfying  \[\lambda(ST) =(x_\mathcal{U}\otimes f_\mathcal{U})(T)=\langle \psi(S) , \phi(T)\rangle, \quad  (S,T\in B(X)).\] So $B(X)$ is Arens regular, as required.
 \end{proof}
Since for every super-reflexive space $X$ the algebra $B(X)$ is Arens regular (see  \cite[Theorem 1]{D}), we also get the next corollary.
\begin{corollary}\label{su}
Every super-reflexive space is ultra-reflexive.
\end{corollary}
To the best of our knowledge, we do not know an ultra-reflexive space which is not super-reflexive!
It would be much desirable if one can provide  such an example. An example of a reflexive space which is not ultra-reflexive can be found in \cite[Corollary 2]{D}.
\section{On Daws's Conjecture}\label{6}
Daws  showed that, if $X$ is super-reflexive then  $B(X)$ is Arens regular; \cite[Theorem 1]{D}.
  He also conjectured  for the accuracy of the converse. 
  The following incomplete idea may lead the reader to provide a decision for the converse.

   Let $B(X)$ be Arens regular  and let  $\mathcal{U}$ be an arbitrary ultrafilter. For $x_\mathcal{U}\in X_\mathcal{U}$ choose $f_{x_\mathcal{U}}\in {X_\mathcal{U}}^*$ so that $\|f_{x_\mathcal{U}}\|=1$ and $f_{x_\mathcal{U}}(x_\mathcal{U})=\|x_\mathcal{U}\|.$ It induces a functional $\lambda_{x_\mathcal{U}}: B(X)\rightarrow\mathbb C$ defined by $\langle\lambda_{x_\mathcal{U}}, T\rangle=\langle f_\mathcal{U}, T(x_\mathcal{U})\rangle.$ Then, by Proposition \ref{arens},   there exist a reflexive space $Z_{x_\mathcal{U}}$ and the operators $\phi_{x_\mathcal{U}} :B(X)\to Z_{x_\mathcal{U}}$ and $\psi_{x_\mathcal{U}} :B(X)\to {Z_{x_\mathcal{U}}}^*$ such that $\|\phi_{x_\mathcal{U}} \|\leq\| \lambda_{x_\mathcal{U}}\|$ and $\langle{\lambda_{x_\mathcal{U}}}, ST\rangle= \langle \psi_{x_\mathcal{U}} (S),\phi_{x_\mathcal{U}}(T)\rangle$, for all $S,T\in B(X).$ Set $Z={\bigoplus}^{\ell^2}_{x_\mathcal{U}\in X_\mathcal{U}} Z_{x_\mathcal{U}};$ then trivially $Z$ is reflexive. If one  could establish an (isometric) embedding from $ X_\mathcal{U}$ into  $Z$, then the reflexivity of $Z$ implies that $X$ is super-reflexive.  We, however, do not know such  an embedding!  
   
To provide such an embedding,    one may consider the map  $\theta : X_\mathcal{U}\to Z$ which is defined by the rule $\theta (x_\mathcal{U})= \phi_{x_\mathcal{U}}(I)$. Then $\theta$ preserve the norm. Indeed,
 \[
\| \theta (x_\mathcal{U})\|=\|  \phi_{x_\mathcal{U}}(I) \|\leq\|\phi_{x_\mathcal{U}} \|\leq\| \lambda_{x_\mathcal{U}}\|\leq \| x_\mathcal{U}\|\quad {\rm and}
\]
\[\| \theta (x_\mathcal{U})\|=\|  \phi_{x_\mathcal{U}}(I) \|=\sup_{\| z^* \|\leq 1, z^*\in Z^*} |\langle z^*, \phi_{x_\mathcal{U}}(I) \rangle |\geq  |\langle \psi_{x_\mathcal{U}}(I), \phi_{x_\mathcal{U}}(I) \rangle |=|\langle{\lambda_{x_\mathcal{U}}}, I\rangle|=\|x_\mathcal{U}\|.
\]
However,   we  know noting about the linearity of $\theta$.
\section*{acknowledgment}
 We would like to thank Professor Matthew Daws for his very valuable comments  which improved an earlier version of this paper.

\end{document}